\documentclass{amsart}
\usepackage{cite}
\usepackage[all]{xy}
\usepackage{graphicx}
\usepackage{amssymb}
\usepackage{mathrsfs}
\usepackage{extarrows}
\usepackage{verbatim}
\usepackage[numbers,sort&compress]{natbib} % ʹ�ο������������ŵ�������ѹ����ʽ����
\usepackage[bookmarksnumbered, bookmarksopen,
colorlinks,citecolor=blue,linkcolor=blue]{hyperref}
\newtheorem{Theorem}{Theorem}[section]
\newtheorem{Lemma}[Theorem]{Lemma}
\theoremstyle{Definition}
\newtheorem{Definition}[Theorem]{Definition}
\newtheorem{Example}[Theorem]{Example}

\newtheorem{Corollary}[Theorem]{Corollary}

\theoremstyle{Remark}
\newtheorem{Remark}[Theorem]{Remark}

\theoremstyle{notation}

\numberwithin{equation}{section}

 \def\3{\operatorname{3}}
\begin{document}

\title{ Reconstruction of mapping spaces by inverse limits}

\author{Jing-Wen Gao   }
\address{School of Mathematics and Statistics, Huazhong University of Science and Technology, Wuhan 430074, P.R.China}

\email{d202280005@hust.edu.cn}
%    \thanks will become a 1st page footnote.
%\thanks{The first author was supported in part by NSF Grant \#000000.}

\author{Xiao-Song Yang  }
%\address{}
%\email{}

\thanks{} %and the Scientific Research Foundation for the Returned Overseas Chinese Scholars, State Education Ministry.}

%    General info
\subjclass[2020]{54C60, 54D10, 55U05}

\date{\today}

%\dedicatory{This paper is dedicated to our advisors.}

\keywords{finite topological spaces, compact metric spaces, mapping spaces, inverse limits}

\begin{abstract}
	Extending the results of reconstruction of compact metric spaces by inverse limits,
	we show that if $(X, d), (Y, d)$ are compact metric spaces, then the mapping space $Y^X$ is homotopy equivalent to the inverse limit of an inverse system of finite $T_0$-spaces which depends only on the finite open covers of $X$ and $Y$. Applying our tools, we obtain that if $H$ is an isotopy of a compact metric space $(X, d)$, then $H_1H^{-1}_0$ can be approximated in terms of moves of a finite $T_0$-space.
\end{abstract}

\maketitle
\section{Introduction}
Finite topological spaces seem to be uninteresting at first glance. This is based on a mistaken understanding that a finite space is always equipped with the discrete topology.
 In fact,  there is a well-known bijective correspondence between topologies and preorders on a finite set \rm{\cite{PSA1937}}. A one-to-one correspondence up to weak homotopy equivalence between compact polyhedra and finite topological spaces was established by McCord \rm{\cite{MC1966}}.
Stong \rm{\cite{ST1966}}  developed the combinatorial methods in the theory of finite topological spaces.

The technique of inverse limits together with  finite topological spaces has been widely used in reconstructing compact spaces. 
 For example, Clader \rm{\cite{EC2009}} proved that the geometric realization $|K|$ of a finite simplicial complex $K$ has the same homotopy type of the inverse limit of an inverse sequence of finite $T_0$-spaces and the inverse limit contains a homeomorphic copy of $|K|$. Later in \rm{\cite{DMR2015}}, Ruiz generalized Clader's result to compact metric spaces. A similar result was obtained by Bilski in \rm{\cite{PB2017}} for locally compact, paracompact, Hausdorff spaces.

The approximation of compact metric spaces by inverse limits plays an important role in shape theory. This theory, introduced by Borsuk \rm{\cite{KB1968}}, is a modification of the classical homotopy theory for general topological spaces that have bad local behavior.
The relationship between shape theory and the theory of finite topological spaces has been exploited in some recent results about the reconstruction of compact metric spaces \rm{\cite{PJC2022,DM2020}}. Moreover, in \rm{\cite{DM2020}}, the reconstruction of topological properties of compact metric spaces is used to provide a different way to obtain persistent modules \rm{\cite{AZ2005}} from a point cloud, which is an important notion in topological data analysis \rm{\cite{HE2008}}.

In this paper, we generalize the results of the approximation of compact metric spaces
to the mapping spaces of continuous mappings between compact metric spaces. In Theorem \ref{THE1}, it is shown that if $(X, d), (Y, d)$ are compact metric spaces, then the mapping space $Y^X$ is homeomorphic to a subspace of the inverse limit of an inverse system of finite $T_0$-spaces that is a strong deformation retract of this inverse limit. 
Applying McCord's correspondence, we obtain an inverse system of compact polyhedra of which inverse limit is weak homotopy equivalent to the mapping space $Y^X$ in Corollary \ref{CO1}.

This paper is organized as follows. In Section \ref{section2}, we gives some definitions and results concerning finite topological spaces. In Section \ref{section3}, we describe the construction of finite models needed for Theorem \ref{THE1}. In Section \ref{section4}, we prove Theorem \ref{THE1} and Corollary \ref{CO1}. In Section \ref{section5}, using our techniques, we relate the isotopy of a compact metric space to the one of a finite $T_0$-space in Theorem \ref{thmiso}.

\section{Preliminaries}\label{section2}
In this section, we recall some properties of finite topological spaces  that are needed in the sequel. More details can be found in \rm{\cite{JAB2011, JAB2008, JPM1, JPM2}}.

A topological space is called an Alexandroff space if every intersection of open sets is open. If $X$ is an Alexandroff space, for every $x\in X$, $U_x$ denotes  the minimal open set which contains $x$, that is, the intersection of all the open sets of $X$  containing $x$. 
The family $\{U_x\}_{x\in X}$ forms the minimal basis for the topology of $X$. 
Given an Alexandroff space, we define a reflexive and transitive relation (preorder) on $X$ with   $x\leq y$ provided $U_x\subseteq U_y$. This relation is a partial order  if and only if $X$ is $T_0$. Conversely, if $X$ is a set with a reflexive and transitive relation on it, then the sets $U_x=\{y\in X~|~y\leq x\}$ forms a basis for an Alexandroff topology on $X$. 
This Alexandroff topology on $X$ is $T_0$ if and only if this relation is a partial order.
A finite $T_0$-space is a finite Alexandroff $T_0$-space.

Hence, for a given set $X$, the discussion above gives a one-to-one correspondence between
Alexandroff topologies on it and its reflexive, transitive relations. Moreover, we can see that Alexandroff $T_0$-spaces and posets (a set with a partial order) are the same object. The notions of morphisms in Alexandroff spaces and in posets are exactly the same. More precisely, a  map between  Alexandroff spaces is continuous if and only if it is order preserving, namely, $x\leq y$ implies $f(x)\leq f(y)$.
 
With the equivalence discussed above, we will relate finite $T_0$-space with finite simplicial complexes by two theorems, following the work of McCord \rm{\cite{MC1966}}. We write |K| for the geometric realization of a simplicial complex $K$.

\begin{Theorem} {\rm (McCord)}\label{MC1}
Let $X$ be a finite $T_0$-space. There exists a finite simplicial complex $\mathscr{K}(X)$, and a weak homotopoy equivalence $$h_X: |\mathscr{K}(X)|\rightarrow X.$$ A continuous map $\phi: X\rightarrow Y$ of finite $T_0$-spaces induces a simplicial map $\mathscr{K}({\phi}): \mathscr{K}(X)\rightarrow \mathscr{K}(Y)$ such that $\phi\circ h_X=h_Y\circ|\phi|$.
\end{Theorem}
In fact, $\mathscr{K}(X)$ is a simplicial complex whose vertices are the points of $X$, and simplices are the subsets of $X$ that are totally ordered.

\begin{Theorem}{\rm (McCord)}
	Let  $K$ be  a finite simplicial complex. There exists a finite $T_0$-space $\chi(K)$,
 and  a weak homotopy equivalence $$h_K: |K|\rightarrow \chi(K).$$
	A simplicial map $\psi: K\rightarrow L$ induces a continuous map $\chi({\psi}): \chi(K)\rightarrow \chi(L)$ such that $\chi(\psi)\circ h_K$ is homotopic to $h_L\circ|\psi|$.
\end{Theorem}
Let $K^{\prime}$ be the barycentric subdivision of $K$. $\chi(K)$ is a set whose points are the vertices of $K^{\prime}$, and the partial order is given by $b(\sigma)\leq b(\sigma^{\prime})$ provided $\sigma\subseteq\sigma^{\prime}$, where $b(\cdot)$ denotes the barycenter of the corresponding simplex. It is not hard to see that $\mathscr{K}(\chi(K))=K^{\prime}$.

\section{Construction of the Finite Models }\label{section3}
Let $(X, d)$ be a compact metric space. We first construct a family of quotient spaces of $X$ that is similar to the one in \rm{\cite[p.212]{PB2017}}.

For every $x\in X$ and a positive real number $\epsilon$, $B(x, \epsilon)$ denotes  the open ball containing $x$ with radius $\epsilon$.
Let $\mathcal{U}_n (n\in\mathbb{N}, ~n\geq 1)$ be a finite open cover of $X$ given by the open balls with radius $\frac{1}{n}$. Define a finite open cover $\mathcal{W}_n$ by $$\mathcal{W}_n=\bigcup_{i=1}^n\mathcal{U}_i.$$
Let $I_X$ be the set of all finite covers $\mathcal{W}_n$. 
We define a function $R: I_X\rightarrow \mathbb{R}$ by $R(\mathcal{W}_n)=\frac{1}{n}$.
The set $I_X$ is directed by the inclusion relation, that is, for ${\alpha}, {\beta}\in I_X,$ ${\alpha}\leq {\beta}$ if ${\alpha}\subseteq{\beta}$.

Given a finite cover $\alpha\in I_X$, we consider the set $$U_x^{\alpha}=\bigcap_{x\in U, U\in {\alpha}}U.$$ 
It is clear that $U_x^{\alpha}$ is open. We define a relation $\sim_{\alpha}$ on $X$ as follows. $$x\sim_{\alpha}y \quad\text{if and only if}\quad U_x^{\alpha}=U_y^{\alpha}.$$
This is obviously an equivalence relation. Denote by $X_{\alpha}$ the quotient space with respect to the relation $\sim_{\alpha}$.
If ${\alpha}\leq {\alpha^{\prime}}$, then by the definitions of $\sim_{\alpha}$ and $\sim_{\alpha^{\prime}}$, each class $E$ in $X_{\alpha^{\prime}}$ is contained in a unique class in $X_{\alpha}$. This gives rise to a collection of maps $$\{f_{\alpha}^{\alpha^{\prime}}: X_{\alpha^{\prime}}\rightarrow X_{\alpha},~\alpha\leq \alpha^{\prime}\in I_X\}.$$

Note that each class $E_{\alpha}\in X_{\alpha}$ is a subset of $X$. We write $\overline{{E}_{\alpha}}$ for the closure of $E_{\alpha}$ in $X$.

\begin{Remark}\label{RK1}
{\rm{If $E_{\alpha}\in X_{\alpha}$, then $E_{\alpha}$ is the disjoint union of all classes of $(f_{\alpha}^{\alpha^{\prime}})^{-1}(E_{\alpha})$ whenever $\alpha\leq \alpha^{\prime}$.}}
\end{Remark}
\begin{Remark}{\rm{
	Let $\prod\limits_{\alpha\in I_X} E_{\alpha}\in \prod\limits_{\alpha\in I_X}  X_{\alpha}$ with $f_{\alpha}^{\alpha^{\prime}}(E_{\alpha^{\prime}})=E_{\alpha}$ whenever $\alpha\leq\alpha^{\prime}$.
	By the same argument in  \rm{\cite[p.213]{PB2017}}, $\bigcap_{\alpha\in I_X}\overline{{E}_{\alpha}}$ is non-empty and contains a unique point of $X$.}}
	
\end{Remark}

\begin{Example}\label{EX1}{\rm{
	Let  $(X, d)=[0, 1]$ with the usual Euclidean metric. The index set $I_X$ is given by $\mathcal{U}_n=\{B(\frac{k}{n}, \frac{1}{n}),~0\leq k\leq n\}$.
	Let
	$$\alpha=\mathcal{W}_2=\{B(0, 1)~ B(1, 1),~B(0,\frac{1}{2}),~B(\frac{1}{2},~\frac{1}{2})~,  B(1,\frac{1}{2})\}.$$ Then, $B(0, 1)=[0,1),~B(1, 1)=(0, 1],$ $B(0,\frac{1}{2})=[0, \frac{1}{2})$, $B(\frac{1}{2},~\frac{1}{2})=(0, 1),~B(1,\frac{1}{2})=(\frac{1}{2}, 1]$. Thus we have that
	\begin{align*}
	&U_0^{\alpha}=[0,1)\cap [0, \frac{1}{2})=[0, \frac{1}{2}).\\
	&U_x^{\alpha}=[0,1)\cap (0, 1]\cap [0, \frac{1}{2})\cap (0, 1)=(0, \frac{1}{2}),~\text{if $0<x<\frac{1}{2}$}.\\
	&U_{\frac{1}{2}}^{\alpha}=[0,1)\cap (0, 1]\cap (0, 1)=(0, 1).\\
	&U_x^{\alpha}=[0,1)\cap (0, 1]\cap (0, 1)\cap (\frac{1}{2}, 1]=(\frac{1}{2}, 1),~\text{if $\frac{1}{2}<x<1$}.\\
	&U_1^{\alpha}=(0, 1]\cap (\frac{1}{2}, 1]=(\frac{1}{2}, 1].
	\end{align*}
	It follows that $X_{\alpha}=\{0,~ (0, \frac{1}{2}),~ \frac{1}{2},~ (\frac{1}{2}, 1), 1\}$.
	
	In fact, if $\alpha=\mathcal{W}_n=\cup_{i=1}^{n}\mathcal{U}_i$, after the same calculation as above, we get 
	\begin{equation}\label{formula1}
	X_{\alpha}=\{ \frac{t}{n}, ~ (\frac{k}{n}, \frac{k+1}{n}),~ 0\leq t\leq n, ~ 0\leq k\leq n-1\}.
	\end{equation}}}
\end{Example}

Let $(X, d),~ (Y, d)$ be compact metric spaces.
Consider a directed set $$I_Y^X=\{\beta^{\alpha}~|~\alpha\in I_X, \beta\in I_Y\},$$ which is directed as follows.  $\beta^{\alpha}\leq \bar\beta^{\bar\alpha}$ if $\alpha\leq\bar\alpha$ and $\beta\leq\bar\beta.$ 
 For $\alpha\in I_X, \beta\in I_Y$, let us consider the finite set
$$Y_{\beta}^{X_{\alpha}}=\{f~|\text{~set maps}~f: X_{\alpha}\rightarrow Y_{\beta}\}.$$
The power set $2^{Y_{\beta}^{X_{\alpha}}}$ is endowed with the topology generated by the sets 
$$B_S=\{T\in 2^{Y_{\beta}^{X_{\alpha}}} | S\subseteq T\},~ S\in 2^{Y_{\beta}^{X_{\alpha}}},$$
that is, $\{B_S,~S\in 2^{Y_{\beta}^{X_{\alpha}}}\}$ forms a basis for this topology.
We can see that $2^{Y_{\beta}^{X_{\alpha}}}$ with this topology is a finite $T_0$-space.
\begin{Remark}{\rm{
For $S, T\in 2^{Y_{\beta}^{X_{\alpha}}}$, $T\in B_S$ if and only if $S\subseteq T$.
The partial order on $2^{Y_{\beta}^{X_{\alpha}}}$ induced by this minimal basis is that $T\leq S$ provided $B_T\subseteq B_S$ (equivalently $T\in B_S$).}}
\end{Remark}

Let $Y^X$ be the set of continuous maps from $X$ to $Y$ and the topology on $Y^X$ is the compact-open topology. $Y^X$ is called the mapping space from $X$ to $Y$.
 Given a compact subset $K$ of $X$ and an open subset $N$ of $Y$, let $V(K, N)$ denote the set of all continuous maps $f: X\rightarrow Y$ such that $f(K)\subseteq N$. 

Fix a point $y_0\in Y$. $C_{y_0}\in Y^X$ denotes the constant map  sending all elements of $X$ to $y_0$.
For every  $f\in Y^X$, we assign it to an element $S\in 2^{Y_{\beta}^{X_{\alpha}}}$  as follows. If $f=C_{y_0}$, then $S=\emptyset\in 2^{Y_{\beta}^{X_{\alpha}}}$. If $f\not=C_{y_0}$, then
$S$ is the collection of set maps $g: X_{\alpha}\rightarrow Y_{\beta}$ with the property that for each class $E_{\alpha}\in X_{\alpha}$, $g(E_{\alpha})\in Y_{\beta}$ belongs to the set
\begin{align}\label{D1}
\{F_{\beta}\in Y_{\beta}~|~\overline{f(E_{\alpha})}\cap\overline{F_{\beta}}\not=\emptyset\}.
\end{align}
It is obvious that such $g(E_{\alpha})$ exists, since for every $x\in E_{\alpha},~ [f(x)]_{\beta}\in(\ref{D1})$. 
This yields a  map $p_{\beta}^{\alpha}: Y^{X}\rightarrow 2^{Y_{\beta}^{X_{\alpha}}}$.

\begin{Example}\label{EX2}{\rm{
	Let $(X, d)=(Y, d)=[0,1]$ with the usual Euclidean metric and let $I_X, I_Y$ be given by the same $\mathcal{U}_n$ as in Example \ref{EX1}.
	Let us select
	 $\alpha=\mathcal{W}_1\in I_X,~ \beta=\mathcal{W}_1\in I_Y$. Then
	$X_{\alpha}=Y_{\beta}=\{0,~ (0, 1),~1\}$. Let $f$ be the identity map. Consider $p_{\beta}^{\alpha}: Y^{X}\rightarrow 2^{Y_{\beta}^{X_{\alpha}}}$.
We will write the image $p_{\beta}^{\alpha}(f)$ explicitly.  According to (\ref{D1}), for any $g\in p_{\beta}^{\alpha}(f)$,
	\begin{align*}
	&g(0)\in\{0,~ (0, 1)\}.\\
	&g((0, 1))\in\{0,~ (0, 1),~1\}.\\
	&g(1)\in\{ (0, 1),~1\}.
	\end{align*}
	Thus  $p_{\beta}^{\alpha}(f)$ has 12 elements:
\begin{align*}
&g_1(0)=0, & &g_1((0, 1))=0, & &g_1(1)=(0, 1). \\
&g_2(0)=0, & &g_2((0, 1))=0, & &g_2(1)=1.\\
&g_3(0)=0, &&g_3((0, 1))=(0, 1),  &&g_3(1)=(0, 1).\\
 &g_4(0)=0, &&g_4((0, 1))=(0, 1),  &&g_4(1)=1.\\
&g_5(0)=0, &&g_5((0, 1))=1,  &&g_5(1)=(0, 1). \\
&g_6(0)=0, &&g_6((0, 1))=1,  &&g_6(1)=1.\\
&g_7(0)=(0, 1), &&g_7((0, 1))=0,  &&g_7(1)=(0, 1).\\
 &g_8(0)=(0, 1), &&g_8((0, 1))=0,  &&g_8(1)=1.\\
&g_9(0)=(0, 1), &&g_9((0, 1))=(0, 1), &&g_9(1)=(0, 1). \\
&g_{10}(0)=(0, 1), &&g_{10}((0, 1))=(0, 1), &&g_{10}(1)=1.\\
&g_{11}(0)=(0, 1), &&g_{11}((0, 1))=1, &&g_{11}(1)=(0, 1).\\
 &g_{12}(0)=(0, 1), &&g_{12}((0, 1))=1, &&g_{12}(1)=1.
\end{align*}}}
\end{Example}

\begin{Lemma}\label{1.2}
	$p_{\beta}^{\alpha}: Y^{X}\rightarrow 2^{Y_{\beta}^{X_{\alpha}}}$ is continuous.
\end{Lemma}
\begin{proof}
	Let $p_{\beta}^{\alpha}(f)=S$. To show that $p_{\beta}^{\alpha}$ is continuous at the point $f$, let us consider the minimal open set $B_{S}$ of $S$.
	If $f=C_{y_0}$, then $B_S=B_{\emptyset}=2^{Y_{\beta}^{X_{\alpha}}}$, $\emptyset\in p_{\beta}^{\alpha}(Y^{X})\subseteq B_{\emptyset}$. Thus $p_{\beta}^{\alpha}$ is continuous at the point $C_{y_0}$. Now we consider the case $f\not=C_{y_0}$.
	 Since $X_{\alpha}$ is a finite topological space, we can suppose that $X_{\alpha}=\{E^1_{\alpha},\cdots, E^n_{\alpha}\}$. Given any $g\in S$, let $n^i_g$ be the minimal positive real number of
	 $$\{d({g(E^i_{\alpha})},  {{F}_{\beta}}), F_{\beta}\in Y_{\beta} \}.$$
	 If no such $n^i_g$ exists,   define $n^i_g=1$. The choice of $g(E^i_{\alpha})$ allows us to choose a non-empty compact set
 $K_g^i\subseteq f^{-1}(B({g(E^i_{\alpha})},n^i_g))\cap  \overline E_{\alpha}^i$.
  
	  Define 
	$$ V=\bigcap_{1\leq i\leq n,~g\in S} V(K^i_g, B({g(E^i_{\alpha})},n^i_g)).$$
	We claim that  $V$ is the desired open neighborhood of $f$ that satisfies $$S\in p_{\beta}^{\alpha}(V)\subseteq B_{S}.$$ It suffices to show that for any given $h\in V$, $S\subseteq p_{\beta}^{\alpha}(h)$. Given a map $g\in S$, from the construction of $V$, we see that  there exists a point $x\in K^i_g\subseteq \overline E^i_{\alpha}$ with 
	$$ h(x)\in B({g(E^i_{\alpha})},n^i_g).$$
	 Thus by the choice of $n_g^i$, $d([h(x)]_{\beta}, {g(E^i_{\alpha})})=0$, and then $\overline{[h(x)]_{\beta}}\cap \overline{g(E^i_{\alpha})}\not=\emptyset$, which implies
	  $g\in p_{\beta}^{\alpha}(h)$. 
\end{proof}

 We shall construct a collection of continuous maps 
 $f^{\tilde\alpha\tilde \beta}_{\alpha\beta}: 2^{Y_{\tilde\beta}^{X_{\tilde\alpha}}}\rightarrow 2^{Y_{\beta}^{X_{\alpha}}}$ whenever $\beta^{\alpha}\leq \tilde\beta^{\tilde\alpha}\in I^X_Y$. For each $S_{\tilde\beta}^{\tilde\alpha}\in 2^{Y_{\tilde\beta}^{X_{\tilde\alpha}}}$, 
 we assign it to an element $S_{\beta}^{\alpha}$ of $2^{Y_{\beta}^{X_{\alpha}}}$ as follows.
 If $S_{\tilde\beta}^{\tilde\alpha}=\emptyset$, then $S_{\beta}^{\alpha}=\emptyset$.
 If $S_{\tilde\beta}^{\tilde\alpha}\not=\emptyset$, then
 $S_{\beta}^{\alpha}$ consists of set maps $g: X_{\alpha}\rightarrow Y_{\beta}$ satisfying that for each class $E_{\alpha}$,  the possible values of $g(E_{\alpha})$ are given by the  set
 \begin{align}\label{D2}
 &\left.\{F_{\beta}\in Y_{\beta}~|~\text{there exists a class $E_{\tilde\alpha}\in (f^{\tilde\alpha}_{\alpha})^{-1}(E_{\alpha})$ and $\tilde{h}\in S_{\tilde\beta}^{\bar\alpha}$ such that}\right.\\ &\left. \tilde{h}(E_{\tilde\alpha})\in(f^{\tilde\beta}_{\beta})^{-1}(F_{\beta}),~\text{that is}, ~ \tilde{h}(E_{\tilde\alpha})\subseteq F_{\beta} \right.\}.\nonumber
 \end{align}

\begin{Lemma}
 $f^{\tilde\alpha\tilde \beta}_{\alpha\beta}: 2^{Y_{\tilde\beta}^{X_{\tilde\alpha}}}\rightarrow 2^{Y_{\beta}^{X_{\alpha}}}$  is continuous.
\end{Lemma}
\begin{proof}
Let $f^{\tilde\alpha\tilde \beta}_{\alpha\beta}(S_{\tilde\beta}^{\tilde\alpha})=S_{\beta}^{\alpha}$. Consider the minimal open set $B_{S_{\beta}^{\alpha}}$  of $S_{\beta}^{\alpha}$.
If $S_{\tilde\beta}^{\tilde\alpha}=\emptyset$, $\emptyset\in f^{\tilde\alpha\tilde \beta}_{\alpha\beta}(B_{\emptyset})= f^{\tilde\alpha\tilde \beta}_{\alpha\beta}(2^{Y_{\tilde\beta}^{X_{\tilde\alpha}}})\subseteq 2^{Y_{\beta}^{X_{\alpha}}}= B_{\emptyset}$. Then $f^{\tilde\alpha\tilde \beta}_{\alpha\beta}$ is continuous at $\emptyset$.
If $S_{\tilde\beta}^{\tilde\alpha}\not=\emptyset$,
 for any $\widetilde{T}\in B_{S_{\tilde\beta}^{\tilde\alpha}}$, $S_{\tilde\beta}^{\tilde\alpha}\subseteq \widetilde{T}$ means 
$$S_{\beta}^{\alpha}\subseteq f^{\tilde\alpha\tilde \beta}_{\alpha\beta}(\widetilde{T}).$$
Therefore, $S_{\beta}^{\alpha}\in f^{\tilde\alpha\tilde \beta}_{\alpha\beta}(B_{S_{\tilde\beta}^{\tilde\alpha}})\subseteq B_{S_{\beta}^{\alpha}}$.
\end{proof}

Let $W_{\beta}^{\alpha}=p_{\beta}^{\alpha} (Y^{X})$ be a subspace of $2^{Y_{\beta}^{X_{\alpha}}}$.
We will prove that $f^{\tilde\alpha\tilde \beta}_{\alpha\beta}$ can be restricted to a continuous map $f^{\tilde\alpha\tilde \beta}_{\alpha\beta}: W_{\tilde\beta}^{\tilde\alpha}\rightarrow W_{\beta}^{\alpha}$.

\begin{Lemma}\label{CCC}
	$f^{\tilde\alpha\tilde \beta}_{\alpha\beta}(W_{\tilde\beta}^{\tilde\alpha})=W_{\beta}^{\alpha}$.      
\end{Lemma}
 \begin{proof}
 We keep the notations used in (\ref{D1}) and (\ref{D2}) in this proof. Let $S_{\tilde\beta}^{\tilde\alpha}\in W_{\tilde\beta}^{\tilde\alpha}$.
 If $S_{\tilde\beta}^{\tilde\alpha}=\emptyset$, then $f^{\tilde\alpha\tilde \beta}_{\alpha\beta}(\emptyset)=\emptyset\in W_{\beta}^{\alpha}$.
 If $S_{\tilde\beta}^{\tilde\alpha}\not=\emptyset$, then   $p_{\tilde\beta}^{\tilde\alpha}(f)=S_{\tilde\beta}^{\tilde\alpha}$ for some $f\not=C_{y_0}\in Y^X$. 
 Let $p_{\beta}^{\alpha}(f)=S_{\beta}^{\alpha}$. We claim that $f^{\tilde\alpha\tilde \beta}_{\alpha\beta}(S_{\tilde\beta}^{\tilde\alpha})=S_{\beta}^{\alpha}$.  

 For any $g\in S_{\beta}^{\alpha}$, we only need to consider its value at class $E_{\alpha}$. Assume that $g(E_{\alpha})=F_{\beta}$. By Remark \ref{RK1}, we have 
 $$\overline{f((f_{\alpha}^{\tilde\alpha})^{-1}(E_{\alpha}))}\cap (f_{\beta}^{\tilde\beta})^{-1}(\overline{F_{\beta}})\not=\emptyset.$$
 Then there exists a class $E_{\tilde{\alpha}}\in (f_{\alpha}^{\tilde\alpha})^{-1}(E_{\alpha})$ and a class
$ F_{\tilde{\beta}}\in (f_{\beta}^{\tilde\beta})^{-1}(F_{\beta})$ such that $\overline{f(E_{\tilde{\alpha}})}\cap \overline F_{\tilde\beta}\not=\emptyset$.
 This is equivalent to saying that $(f_{\beta}^{\tilde\beta})^{-1}(F_{\beta})$ contains a class of $Y_{\tilde\beta}$ which is the image of a map from $S_{\tilde\beta}^{\tilde\alpha}$ at $E_{\tilde{\alpha}}\in (f_{\alpha}^{\tilde\alpha})^{-1}(E_{\alpha})$.
 Thus $g\in f^{\tilde\alpha\tilde \beta}_{\alpha\beta}(S_{\tilde\beta}^{\tilde\alpha})$. 
 On the other hand, if $g\in f^{\tilde\alpha\tilde \beta}_{\alpha\beta}(S_{\tilde\beta}^{\tilde\alpha})$, then by (\ref{D2}), for each $E_{\alpha}\in X_{\alpha}$, there exists $E_{\tilde{\alpha}}\in (f_{\alpha}^{\tilde\alpha})^{-1}(E_{\alpha})$ and $\tilde{h}\in S_{\tilde\beta}^{\bar\alpha}$ with $\tilde{h}(E_{\tilde\alpha})\subseteq g(E_{\alpha})$.
 Then $\overline{g(E_{\alpha})}\cap \overline{f(E_{\alpha})}\not=\emptyset$ because $\overline{\tilde{h}(E_{\tilde\alpha})}\cap \overline{f(E_{\tilde\alpha})}\not=\emptyset$. Therefore, $g\in S_{\beta}^{\alpha}$.
 
\end{proof}

Indeed, in the proof of Lemma \ref{CCC}, we prove the commutativity of the following diagram.
 \begin{equation*}\label{diagram1}
\xymatrix{Y^X\ar[d]_{p_{\tilde\beta}^{\tilde\alpha}}\ar[dr]^{p_{\beta}^{\alpha}} \\
2^{Y_{\tilde\beta}^{X_{\tilde\alpha}}} \ar[r]^{f_{\alpha\beta}^{\tilde\alpha\tilde\beta}}  & 2^{Y_{\beta}^{X_{\alpha}}}.\\
}
\end{equation*}

We now extend the identity map $id: W_{\beta}^{\alpha}\rightarrow W_{\beta}^{\alpha}$
to a continuous map $r_{\beta}^{\alpha}: 2^{Y_{\beta}^{X_{\alpha}}}\rightarrow W_{\beta}^{\alpha}$.
Let $S\in( W_{\beta}^{\alpha})^c$. 
 $r_{\beta}^{\alpha}(S)$ is defined to be the minimum of the totally ordered set
 $$ 2^S\cap W_{\beta}^{\alpha},$$ 
 that is, the set of maximum cardinality.
  $2^S\cap W_{\beta}^{\alpha}$ is non-empty because $\emptyset\in 2^S\cap W_{\beta}^{\alpha}$.
  \begin{Lemma}
  	$r_{\beta}^{\alpha}: 2^{Y_{\beta}^{X_{\alpha}}}\rightarrow W_{\beta}^{\alpha}$ is continuous.
  \end{Lemma}
\begin{proof}
 To verify the continuity of $r_{\beta}^{\alpha}$, we
consider the minimal open set $B_{r_{\beta}^{\alpha}(S)}\cap W_{\beta}^{\alpha}$ of $r_{\beta}^{\alpha}(S)$. We will prove  $r_{\beta}^{\alpha}(B_S)\subseteq B_{r_{\beta}^{\alpha}(S)}\cap W_{\beta}^{\alpha}$.

Let $S\in( W_{\beta}^{\alpha})^c$.  If $T\in B_S\cap  W_{\beta}^{\alpha}$, then $$r_{\beta}^{\alpha}(S)\subseteq S\subseteq T=r_{\beta}^{\alpha}(T),$$ so $r_{\beta}^{\alpha}(T)\in B_{r_{\beta}^{\alpha}(S)}\cap W_{\beta}^{\alpha}.$ If $T\in B_S\cap  (W_{\beta}^{\alpha})^c$, we have $2^S\cap W_{\beta}^{\alpha}\subseteq 2^T\cap W_{\beta}^{\alpha},$ then $r_{\beta}^{\alpha}(S)\subseteq r_{\beta}^{\alpha}(T),$ so $r_{\beta}^{\alpha}(T)\in B_{r_{\beta}^{\alpha}(S)}\cap W_{\beta}^{\alpha}.$
For $S\in  W_{\beta}^{\alpha}$,~ $r_{\beta}^{\alpha}(S)=S$. If $T\in B_S\cap  W_{\beta}^{\alpha}$, then $r_{\beta}^{\alpha}(T)=T\in B_S\cap W_{\beta}^{\alpha}$. If $T\in B_S\cap  (W_{\beta}^{\alpha})^c$, we have $$S\in 2^T\cap W_{\beta}^{\alpha},$$  then $S\subseteq r_{\beta}^{\alpha}(T)$ and so   $r_{\beta}^{\alpha}(T)\in B_S\cap W_{\beta}^{\alpha}$.
\end{proof}

It is straightforward to check the commutativity of the following diagram.
\begin{equation*}
\xymatrix{  2^{Y_{\tilde\beta}^{X_{\tilde\alpha}}}\ar[r]^{f_{\alpha\beta}^{\tilde\alpha\tilde\beta}}
	\ar[d]_{r_{\tilde\beta}^{\tilde\alpha}} &  2^{Y_{\beta}^{X_{\alpha}}}\ar[d]^{r_{\beta}^{\alpha}}\\
	W_{\tilde\beta}^{\tilde\alpha}\ar[r]^{f_{\alpha\beta}^{\tilde\alpha\tilde\beta}}&W_{\beta}^{\alpha}.
}
\end{equation*}
Hence, we get an inverse system of finite $T_0$-spaces, $$\mathscr{A}=\{2^{Y_{\beta}^{X_{\alpha}}},~ I_Y^X, ~ \{f^{\tilde\alpha\tilde \beta}_{\alpha\beta}, ~ \beta^{\alpha}\leq \tilde\beta^{\tilde\alpha} \} \},$$
and define $\widetilde{Y^X}$ to be its inverse limit. 
Define a  map $p: Y^X\rightarrow \widetilde{Y^X}$ by 
$$p(f)=\prod\limits_{\beta^{\alpha}\in I_Y^X}p_{\beta}^{\alpha}(f).$$
Similarly, the collection $\{r_{\beta}^{\alpha}\}_{\beta^{\alpha}}$ of maps  induces a continuous map $r: \widetilde{Y^X}\rightarrow p(Y^X)$ defined by
$$r(S)=\prod\limits_{\beta^{\alpha}\in I_Y^X}r_{\beta}^{\alpha}(S_{\beta}^{\alpha}),$$
where $S=\prod\limits_{\beta^{\alpha}\in I_Y^X} S_{\beta}^{\alpha}\in \widetilde{Y^X}$.

\begin{Remark}{\rm{
The topology of $\widetilde{Y^X}$ is generated by the sets
$$B_S=\{\l(\prod\limits_{\beta^{\alpha}\in I_Y^X}T_{\beta}^{\alpha})\cap\widetilde{Y^X}~|~S_{\beta}^{\alpha}\subseteq T_{\beta}^{\alpha}\},\quad S=\prod\limits_{\beta^{\alpha}\in I_Y^X} S_{\beta}^{\alpha}\in \widetilde{Y^X}.$$}}
\end{Remark}

\section{ The Main Theorem and its proof}\label{section4}
 \begin{Theorem}\label{THE1}
 	Let $X,Y$ be compact metric spaces. Then $Y^X$ is homotopy equivalent to the inverse limit $\widetilde{Y^X}$.
 	More precisely, $Y^X$ is homeomorphic to the subspace $p(Y^X)\subseteq \widetilde{Y^X}$, as a strong deformation retract of $\widetilde{Y^X}$.
 \end{Theorem}
\begin{proof}
	The continuity of $p$ follows from Lemma \ref{1.2}.  $p$ is clearly injective.  Suppose we have two different maps $f, g$ of $Y^X$. Then there is a point $x\in X$ such that $f(x)\not=g(x)$. One can choose two  open neighborhoods $V_1, V_2$ of $f(x)$ and $g(x)$ respectively with $\overline{V_1}\cap\overline {V_2}=\emptyset$. Let
$$d(\overline{V_1}, \overline {V_2})=\eta>0.$$
The continuity of $f, g$ allows us to find  open neighborhoods $W_1, W_2$ of $x$ with   $f(W_1)\subseteq V_1$ and $g(W_2)\subseteq V_2$. 

Let $\beta\in I_Y$. It is clear that $f(W_1\cap W_2)$ and $g(W_1\cap W_2)$ are covered by finite open balls in $\beta$. As $R(\beta)$ goes to $0$, these open balls are contained in $V_1, V_2$ respectively. Thus there exists ${\beta}\in I_Y$ satisfying $R(\beta)<\frac{\eta}{2}$  and for any $x_1, x_2\in  W_1\cap W_2$, $$U^{\beta}_{f(x_1)}\subseteq {V_1},~U^{\beta}_{g(x_2)}\subseteq {V_2}.$$ 
Consider ${\alpha}\in I_X$ with $$U_x^{\alpha}\subseteq W_1\cap W_2.$$ It follows immediately that $[x]_{\alpha}\subseteq W_1\cap W_2,$
 and $[f(x_1)]_{\beta}\subseteq V_1,~[g(x_2)]_{\beta}\subseteq V_2$ for any $x_1, x_2\in  [x]_{\alpha}.$
 Suppose to the contrary that $p(f)=p(g)$. Then $(p_{\beta}^{\alpha})_{r(\beta)}(f)=(p_{\beta}^{\alpha})_{r(\beta)}(g)$ 
  for $\alpha,$  $\beta$ we have chosen. Note that
for every given map $h\in (p_{\beta}^{\alpha})_{r(\beta)}(f)=(p_{\beta}^{\alpha})_{r(\beta)}(g)$, there exist $z_1, z_2\in {h([x]_{\alpha})}$ and $x_1, x_2\in  [x]_{\alpha}$ such that $$d(z_1, f(x_1))= d(z_2, g(x_2))=0.$$  
Since $\text{diam}({h([x]_{\alpha})})\leq{2}{R(\beta)},$ we have
 $$ d(f(x_1), g(x_2))\leq d(f(x_1), z_1)+d(z_1, z_2)+d(z_z, g(x_2))\leq{2}{R(\beta)}<\eta.$$
 This contradicts with the fact that $f(x_1)\in V_1,~g(x_2)\in V_2.$
 Thus $(p_{\beta}^{\alpha})_{r(\beta)}(f)\not=(p_{\beta}^{\alpha})_{r(\beta)}(g)$,
it follows that $p$ is injective.

We now obtain a well-defined map  $$q: p(Y^X)\rightarrow Y^X,\quad S=\prod\limits_{\beta^{\alpha}\in I_Y^X} S_{\beta}^{\alpha}\mapsto \bigcap_{\beta^{\alpha}\in I_Y^X} (p_{\beta}^{\alpha})^{-1}(S_\beta^{\alpha}),$$
which is the inverse of $p$.

To show that this map is continuous at every point $S$, let us choose an arbitrary open neighborhood $V(K, N)$ of 
$q(S)$, where $K$ is a compact subset of $X$ and $N$ is an open subset of $Y$. For the sake of simplicity,
we shall write $f=q(S)$. We claim that the minimal open neighborhood $$W=\prod\limits_{\beta^{\alpha}\in I_Y^X} B_{S_{\beta}^{\alpha}}\cap W_{\beta}^{\alpha}$$
of $S$ satisfying $q(W)\subseteq V(K, N)$. 

Let $\widetilde{S}=\prod\limits_{\beta^{\alpha}\in I_Y^X} \widetilde S_{{\beta}^{\alpha}}\in W$ and write $\widetilde{f}=q(\widetilde{S})$. This means that for each index $\beta^{\alpha}\in I_Y^X$, 
$$S_{\beta}^{\alpha}\subseteq \widetilde S_{\beta}^{\alpha}.$$
For each $f(x)\in f(K)$, 
there is an open neighborhood $N_{f(x)}$ of $f(x)$ in $N$ satisfying
$ \overline{N}_{f(x)}\subseteq N.$
  Then $\{N_{f(x)}\cap f(K)\}_{x\in K}$ is an open cover of $f(K)$. Since $f(K)$ is compact, we get a finite subcover $\{N_i\cap f(K)\}_{i=1}^n$ such that $f(K)= \cup_{i=1}^{n}(N_i\cap f(K))$. 

Assume to the contrary that $\widetilde{f}\not\in V(K, N)$, which implies there is a point $x\in K$ such that $\widetilde{f}(x)\cap N=\emptyset.$ Then it follows that
$\widetilde{f}(x)\cap(\cup_{i=1}^n\overline{N}_i)=\emptyset.$
One  has an open neighborhood $N_{\widetilde{f}(x)}$ of $\widetilde{f}(x)$ with 
$\overline N_{\widetilde{f}(x)}\cap(\cup_{i=1}^n\overline N_i)=\emptyset.$ 
Let $$d(\overline N_{\widetilde{f}(x)}, ~\cup_{i=1}^n\overline N_i)=\eta>0.$$
By the continuity of $\widetilde{f}$,
we can find an open neighborhood $V_x$ of $x$ with $V_x\subseteq K$ and $\widetilde{f}(V_x)\subseteq N_{\widetilde{f}(x)}$.
Let  $\alpha\in I_X$ with  $U_x^{\alpha}\subseteq V_x$, $\beta\in I_Y$ with $R(\beta)<\frac{\eta}{2}$ and $U_{\widetilde{f}(x_1)}^{\beta}\subseteq N_{\widetilde{f}(x)}, U^{\beta}_{f(x_2)}\subseteq \cup_{i=1}^n{N}_i$ for any $x_1, x_2\in V_x$. A similar argument in the second paragraph of this proof works  for the existence of  $\beta$ in this case.
Note that for every map $h\in (S_{\beta^\prime}^{\alpha^{\prime}})_{r(\beta)}\subseteq  (\widetilde S_{\beta^\prime}^{\alpha^{\prime}})_{r(\beta)}$, one can have points $z_1, z_2\in {h([x]_{\alpha})}$ and $x_1, x_2\in  [x]_{\alpha}$ such that $$d(z_1, \tilde f(x_1))= d(z_2, f(x_2))=0.$$  
 We get
  $$ d(\tilde f(x_1), f(x_2))\leq d(\tilde f(x_1), z_1)+d(z_1, z_2)+d(z_z, f(x_2))\leq{2}{R(\beta)}<\eta.$$
This is a contradiction, because
  $\widetilde{f}(y_1)\in N_{\widetilde{f}(x)},~ f(y_2)\in \cup_{i=1}^n{N}_i.$
Then we conclude that $\widetilde{f}\in V(K, N)$ and $q$ is continuous.

Hence, the composition $\widetilde{Y^X}\xrightarrow{r}p(Y^X)\xrightarrow{q}Y^X$ is a continuous map $qr: \widetilde{Y^X}\rightarrow Y^X$.
Then $p(Y^X)$ is a homeomorphic copy of $Y^X$ inside $\widetilde{Y^X}$. It remains to prove that the map $r: \widetilde{Y^X}\rightarrow p(Y^X)$ satisfies $r\simeq id_{\widetilde{Y^X}}$. Define a homotopy $H: \widetilde{Y^X}\times [0, 1]\rightarrow \widetilde{Y^X}$ by
\begin{equation*}
H(S, t)=
\begin{cases}
S  &\text{if}~t\in[0, 1),\\
r(S) &\text{if}~t=1.
\end{cases}
\end{equation*}
 We only need to prove that $H$ is continuous at the point $(S, 1)$. Consider the minimal open set $B_{r(S)}$ of $r(S)$.
 We claim that $H(B_S\times [0, 1])\subseteq B_{r(S)}$.
 
 For $S\not\in p(Y^X)$, since $r(S)\subseteq S$, $H(B_S\times [0, 1))=B_S\subseteq B_{r(S)}$. If $T\in B_S\times 1\cap (p(Y^X)\times 1)$, then $$r(S)\subseteq S\subseteq T=r(T),$$ so $r(T)\in B_{r(S)}$.
 If $T\in B_S\times 1\cap (p(Y^X))^c\times 1$, we have $2^S\subseteq 2^T$, then $r(S)\subseteq r(T)$, so $r(T)\in B_{r(S)}$. 
 For $S\in p(Y^X)$, $r(S)=S$. Then $H(B_S\times [0, 1))=B_S= B_{r(S)}$. 
If $T\in B_S\times 1\cap ( p(Y^X)\times 1)$, then $S\subseteq T=r(T)$, so $r(T)\in B_{S}$. 
If $T\in B_S\times 1\cap (p(Y^X))^c\times 1$, we have
$$S\in  2^T\cap p(Y^X),$$  then $S\subseteq r(T)$ and so $~r(T)\in B_S$.

\end{proof}

 Theorem \ref{THE1} points out the homotopical structure of $Y^X$ can be recovered by the inverse limit $\widetilde{Y^X}$. As we have seen, the structure of inverse limit $\widetilde{Y^X}$ is much simpler than $Y^X$ and depends only on the finite covers of $X, Y$. 
 
\begin{Example}
{\rm{Consider $(X, d)=(Y, d)=[0, 1]$ with the usual Euclidean metric and the index sets $I_X, I_Y$ given by the same $\mathcal{U}_n$ as in Example \ref{EX1}. Let $\beta^{\alpha}\in I_Y^X.$ Suppose $\alpha=\mathcal{W}_n, \beta=\mathcal{W}_m$, then $X_{\alpha}, X_{\beta}$ are of the form $(\ref{formula1})$. Thus the inverse system $\mathscr{A}$ is easily described. We shall characterize $W_{\beta}^{\alpha}$. Let $S_{\beta}^{\alpha}\in W_{\beta}^{\alpha}$. By a direct calculation, $g\in S_{\beta}^{\alpha}$ if and only if $g$ satisfies the following conditions:
\begin{align*}
(i)~ &g(\frac{t}{n})\in A_{t+1}=\{(\frac{k-1}{m}, \frac{k}{m}), \frac{k}{m}, (\frac{k}{m}, \frac{k+1}{m})\}\cap Y_{\beta}~ \text{or}~ \{(\frac{k-1}{m}, \frac{k}{m})\}\cap Y_{\beta}.\\
(ii)~&g((\frac{t-1}{n}, \frac{t}{n}))\in B_t=\left(\bigcup_{0\leq i\leq s}\{(\frac{k_i-1}{m}, \frac{k_i}{m})\}, \frac{k_i}{m}, (\frac{k_i}{m}, \frac{k_i+1}{m})\right)\cap Y_{\beta} ~\text{or}~ \\
&\{(\frac{k-1}{m}, \frac{k}{m})\}\cap Y_{\beta}, ~\text{where}~ k_0\leq k_1\leq\cdots\leq k_s~\text{or}~k_s\leq k_{s-1}\leq\cdots\leq k_0.\\
(iii)~& A_0\subseteq B_0\supseteq A_1\subseteq B_1\supseteq A_2\subseteq\cdots\supseteq A_n\subseteq B_n\supseteq A_{n+1}.
\end{align*}
Therefore, we can give a characterization of $p(Y^X)$.}}
\end{Example}

It has been shown that once given a collection of nested finite open covers of compact metric spaces $X$ and $Y$, that is, the index set $I_X$ and $I_Y$, one can explicitly compute the inverse system $\mathscr{A}$ defining the inverse limit $\widetilde{Y^X}$ that is homotopy equivalent to the mapping space $Y^X$. 
Observe that, in general, $\widetilde{Y^X}$ is not unique, it depends on the finite open covers we choose at the beginning.

We  now apply Theorem \ref{MC1} to obtain an inverse system of compact polyhedra. For each $ \beta^{\alpha}\in I_Y^X$, there exists a simplicial complex $\mathscr{K}(2^{Y_{\beta}^{X_{\alpha}}})$ and a weak homotopy equivalence
$$h_{\beta}^{\alpha}: |\mathscr{K}(2^{Y_{\beta}^{X_{\alpha}}})|\longrightarrow 2^{Y_{\beta}^{X_{\alpha}}}.$$
Also, when $\beta^{\alpha}\leq\tilde\beta^{\tilde\alpha}\in I_Y^X,$ the map 
$f^{\tilde\alpha\tilde \beta}_{\alpha\beta}$ induces a map $\mathscr{K}(f^{\tilde\alpha\tilde \beta}_{\alpha\beta})$ making the following diagram commute:
\begin{equation*}
\xymatrix{|\mathscr{K}(2^{Y_{\tilde\beta}^{X_{\tilde\alpha}}})|\ar[d]_{h_{\tilde\beta}^{\tilde\alpha}}\ar[r]^{\mathscr{K}(f^{\tilde\alpha\tilde \beta}_{\alpha\beta})}  & |\mathscr{K}(2^{Y_{\beta}^{X_{\alpha}}})|\ar[d]^{h_{\beta}^{\alpha}} \\
	2^{Y_{\tilde\beta}^{X_{\tilde\alpha}}} \ar[r]^{f_{\alpha\beta}^{\tilde\alpha\tilde\beta}}  & 2^{Y_{\beta}^{X_{\alpha}}}.
}
\end{equation*}
 Thus the inverse system $\mathscr{A}$ of finite $T_0$-spaces induces an inverse system 
$$\mathscr{B}=\{|\mathscr{K}(2^{Y_{\beta}^{X_{\alpha}}})|,~ I_Y^X,~  \{\mathscr{K}(f^{\tilde\alpha\tilde \beta}_{\alpha\beta}),~ \beta^{\alpha}\leq\tilde\beta^{\tilde\alpha}\in I_Y^X \}\}$$
of compact polyhedra.
Hence, we have  the following result.

\begin{Corollary}\label{CO1}
	Let $(X, d),~ (Y, d)$ be compact metric spaces. Then the mapping space $Y^X$ is weak homotopy equivalent to the inverse limit of an inverse system of compact polyhedra.
\end{Corollary}

\section{ Isotopies of Compact Metric Spaces}\label{section5}
This section is devoted to providing an approach to the approximation of isotopies of compact metric spaces by means of moves of finite $T_0$-spaces.
We include below some definitions and notations that we will use.
\begin{Definition}{\rm{
Let $M$  be a topological space. An  \emph{isotopy} of $M$  is a level preserving homeomorphism $H: M\times I\rightarrow M\times I$.  $H_t$ is the homeomorphism $M\rightarrow M$ defined by $H(x, t)=(H_tx, t)$ for $(x, t) \in M\times I.$

Two homeomorphisms $f, g: M\rightarrow M$ are isotopic if there exists an isotopy $H$ of $M$ with $H_0=f, H_1=g$.}}
\end{Definition}

Let $X$ be a finite $T_0$-space, $x\in X$. We denote by 
$$U_x=\{y\in X~|~y\leq x\}, \quad F_x=\{y\in X~|~y\geq x\}$$
the minimal open set of $X$ containing $x$ and the closure of the set $\{x\}$ in $X$ respectively.

\begin{Definition}{\rm{
A \emph{move} of a finite $T_0$-space  $X$ is  a homeomorphism of $X$ keeping fixed the complement of a minimal open set $U_x$ for some $x\in X$.}}
\end{Definition}
 
Before establishing the relationship between isotopies of  compact metric spaces and isotopies of  finite $T_0$-spaces, we need a lemma describing the isotopy of a finite $T_0$-space.
\begin{Lemma}\label{isotopy}
Let $H$ be an  isotopy of a finite $T_0$-space $X$. Then $H_1H_0^{-1}$ can be expressed as the composition of a finite number of moves.
\end{Lemma}
\begin{proof}
Note that $X\times I$ has a natural partial order, that is, $(x_1, t_1)\leq (x_2, t_2)$ whenever $x_1\leq x_2$ in $X$ and $t_1\leq t_2$ in $I$.
We first prove that $H$ is order preserving, then so does $H^{-1}$. Suppose $(x_1, t_1)\leq (x_2, t_2)$ in $X\times I$.
 Since $H^{-1}(F_{f(x_1)}\times [t_1, t_2])$ is a closed set containing $(x_1, t_1)$ and $H$ is level preserving,  $$(x_2, t_2)\in F_{x_1}\times [t_1, t_2]\subseteq H^{-1}(F_{f(x_1)}\times [t_1, t_2]).$$ Therefore $H(x_1, t_1)\leq H(x_2, t_2)$. 
	
Let  $q$ be a point of $X$.
Since $X$ is a finite set and $q\times I$ is a totally ordered subset of $X\times I$, we have a sequence of points  $x^q_1,\cdots, x^q_n$ in $X$, such that $x^q_1<\cdots< x^q_n$ and  image of the canonical projection 
$p: H^{-1}(q\times I)\rightarrow X$ is $\{x^q_1,\cdots, x^q_n\}$. Note that $I$ can be partitioned into disjoint connected subsets $I^q_1,\cdots, I^q_n$ satisfying that 
$$ H^{-1}(q, t)=(x^q_i, t),~t\in I^q_i,$$
and $sup I^q_i=inf I^q_{i+1}.$

Consider a subset $W$ of continuous maps from $X$ to $I$ consisting of maps with the property that for each $x\in X$,
$$ (x, f(x))=(x^q_i, sup I^q_i)~\text{or}~(x^q_i, inf I^q_i), ~\text{for some $i$ and $q$}.$$
It is not hard to see that we can choose a sequence of maps $f_0, f_1,\cdots, f_n$ in $W$ such that:
\begin{itemize}
\item[(1)] $f_i$ and $f_{i-1}$ agree on all but a minimal open set $U_x$ for some $x\in X$.
\item[(2)] $f_0\leq\cdots\leq f_n$.
\item [(3)] $f_0(X)=0, f_n(X)=1$
\item[(4)] For any $i$ and $q$, the  graph $f^*_i$ of $f_i$  intersects with $H^{-1}(q\times I)$ at a single point in $X\times I$.
\end{itemize}
 Now let $k_i=pHf^*_i,~ h_i=k_ik_{i-1}^{-1}.$ Then we have $h_0=k_0=H_0$, and
  $$ H_1=h_nh_{n-1}\cdots h_1H_0.$$
\end{proof}

\begin{Theorem}\label{thmiso}
Suppose $(X, d)$ is a compact metric space, and  $H$ is an isotopy of $X$. Given any $\epsilon>0$,  then there exists an index ${\alpha}\in I_{X}$ and a sequence of bijections from $X_{\alpha}$ to $X_{\alpha}$, $h_1,\cdots, h_n$, such that:
\begin{itemize}
\item[(i)] For  $i\geq 1$, $h_i$ is the restriction of some move of $2^{X_{\alpha}}$.
\item [(ii)] For each $E_{\alpha}\in X_{\alpha}$, $d(H_1H^{-1}_0(E_{\alpha}), h_nh_{n-1}\cdots h_1(E_{\alpha}))<\epsilon.$
\end{itemize}
\end{Theorem}
\begin{proof}
Let $\alpha\in I_X$. For each $H_t$,  $p_{\alpha}^{\alpha}(H_t)$, being the image of  $p_{\alpha}^{\alpha}: X^X\rightarrow 2^{X_{\alpha}^{X_{\alpha}}}$, contains a bijection
$X_{\alpha}\rightarrow X_{\alpha}$, which is denoted by $G_t$. Note that
$X_{\alpha}$ is not given any topology, however, $2^{X_{\alpha}}$, when regarded as a subset of $2^{X_{\alpha}^{X_{\alpha}}}$, is a finite $T_0$-space.
Extend $G_t$ to a homeomorphism $2^{X_{\alpha}}\rightarrow 2^{X_{\alpha}}$ by taking
$A$ to $G_t(A)$ for every $A\subseteq X_{\alpha}.$
 Therefore $\{G_t\}_{t\in I}$ combine to give an isotopy of $2^{X_{\alpha}}$ in itself.
 
Applying Lemma \ref{isotopy} to $G$ yields a finite number of moves of $2^{X_{\alpha}}$, $g_1,\cdots, g_n$, such that   
$$ G_1G^{-1}_0=g_ng_{n-1}\cdots g_1.$$
It follows from the definition of $G$ and the construction of $\{g_i\}$ that each $g_i$ preserves cardinalities of all subsets of $X_{\alpha}$, in particular, $g_i$ restricts to a bijection $X_{\alpha}\rightarrow X_{\alpha}$.
If $\alpha$ is chosen with $R(\alpha)$ sufficiently small, then  $h_i,~1\leq i\leq n,$ defined by 
$$ h_i(E_\alpha)=g_i(E_{\alpha}),~\text{for}~E_\alpha\in X_{\alpha},$$
satisfy the required properties.
\end{proof}

\begin{Remark}{\rm{
Some information about $h_i (i\geq 1)$ can  be obtained from the property $(i)$. For example, if $g_i$ keeps some element $A$ of $2^{X_{\alpha}}$ fixed, then so does $h_i$ when $A$ is viewed as a subset of $X_{\alpha}$.}}

\end{Remark}

\section*{Declarations}
The authors declare that they have no conflict of interest.

\end{document}